\documentclass[10pt]{amsart}
\usepackage{amsfonts,amssymb,amscd,amsmath,enumerate,verbatim,calc}

\def\QQ{{\mathbb Q}}

\def\frk{\mathfrak}

\def\mm{{\frk m}}

\def\Phi{{\frk N}}
\def\opn#1#2{\def#1{\operatorname{#2}}} 
\def\opn#1#2{\def#1{\operatorname{#2}}} 
\opn\chara{char} 
\opn\length{\ell}
\opn\pd{pd} 
\opn\rk{rk}

\opn\projdim{proj\,dim}
\opn\injdim{inj\,dim}
\opn\rank{rank}
\opn\depth{depth}
\opn\grade{grade}
\opn\hei{ht}
\opn\embdim{emb\,dim}
\opn\codim{codim}
\opn\Tr{Tr} 
\opn\bigrank{big\,rank}
\opn\superheight{superheight}
\opn\lcm{lcm}
\opn\trdeg{tr\,deg}
\opn\reg{reg} 
\opn\lreg{lreg} 
\opn\ini{in} 
\opn\lpd{lpd}
\opn\size{size}
\opn{\mult}{mult}
\opn\div{div} \opn\Div{Div} \opn\cl{cl} \opn\Cl{Cl}
\opn\Spec{Spec} \opn\Supp{Supp} \opn\supp{supp} 
\opn\Sing{Sing} \opn\Ass{Ass} \opn\Min{Min}
\opn\Ann{Ann} \opn\Rad{Rad} \opn\Soc{Soc}
\opn\Syz{Syz} \opn\Im{Im} \opn\Ker{Ker} \opn\Coker{Coker}
\opn\Am{Am} \opn\Hom{Hom} \opn\Tor{Tor} \opn\Ext{Ext}
\opn\End{End} \opn\Aut{Aut} \opn\id{id}

\opn\nat{nat}
\opn\pff{pf} 
\opn\Pf{Pf} \opn\GL{GL} \opn\SL{SL} \opn\mod{mod} \opn\ord{ord}
\opn\Gin{Gin}
\opn\Hilb{Hilb}
\opn\adeg{adeg}
\opn\std{std}\opn\ip{infpt}
\opn\Pol{Pol}
\opn\sat{sat}
\opn\Var{Var}
\opn\aff{aff} \opn\con{conv} \opn\relint{relint} \opn\st{st}
\opn\lk{lk} \opn\cn{cn} \opn\core{core} \opn\vol{vol}
\opn\link{link} \opn\star{star}
\opn\gr{gr}
\def\Rees{{\mathcal R}}

\def\pot#1#2{#1[\kern-0.28ex[#2]\kern-0.28ex]}

\opn\dirlim{\underrightarrow{\lim}}
\opn\inivlim{\underleftarrow{\lim}}

\newtheorem{Theorem}{Theorem}[section]
\newtheorem{Lemma}[Theorem]{Lemma}
\newtheorem{Corollary}[Theorem]{Corollary}
\newtheorem{Proposition}[Theorem]{Proposition}

\newtheorem{Example}[Theorem]{Example}

\newtheorem{Definition}[Theorem]{Definition}

\let\epsilon\varepsilon
\let\phi=\varphi
\let\kappa=\varkappa
\let\lm=\lambda

\opn\dis{dis}
\opn\Lex{Lex}
\begin{document}
\title[Hilbert Coefficients and depth of $G(I)$]
{Hilbert coefficients and depth of the associated graded ring of an ideal}
\author{J. K. Verma}
\address{J. K. Verma, Department of Mathematics, Indian Institute
of Technology Bombay, Mumbai, India}
\email{jkv@math.iitb.ac.in}

\begin{abstract} In this expository paper we survey results proved during 
the last fifty years  that relate
Hilbert coefficients $e_0(I)$ and $e_1(I)$  
of an $\mm $-primary  ideal $I$ in a Cohen-Macaulay
local ring $(R,\mm )$  with depth of the associated graded ring $G(I).$
Several results in this area follow from two theorems   of S. Huckaba 
and  T. Marley. These were proved using homological techniques. We 
provide simple proofs using  superficial sequences.
\end{abstract}
\maketitle
\thispagestyle{empty}

\section{\bf Introduction} 
Throughout these notes,  $(R,\mm)$ denotes  a Noetherian local ring of 
dimension $d$ and $I$ denotes an $\mm $-primary ideal of $R.$ 
Let $\lm(M)$  denote length of an $R$-module $M.$ 
The Hilbert function 
$H_I(n)$ of $I$ is defined as
$H_I(n)=\lm(R/I^n).$ It is well-known that $H_I(n)$ is a  polynomial function of $n$
of degree $d.$  In other words, 
there is a polynomial  $P_I(x) \in   \QQ[x]$ such that $H_I(n)=P_I(n)$ for 
all large $n.$  It is written in terms  of the binomial coefficients as:
$$ P_I(x)=e_0(I)\binom{x+d-1}{d}-e_1(I)\binom{x+d-2}{d-1}+
\cdots+(-1)^de_d(I)$$  
where $e_i(I)$ for $i=0,1, \ldots,d$ are integers, 
called the Hilbert coefficients of $I.$ The leading coefficient $e_0(I),$ 
called the multiplicity of $I,$ is well-understood. However, not much is 
known about  other coefficients.

The associated graded ring of $I$ is defined to be the graded ring 
$G(I)=\bigoplus_{n=0}^{\infty} I^n/I^{n+1}.$ The objective of this  paper 
is to survey known results that link depth of $G(I)$ with linear relations
among the Hilbert coefficients $e_0(I)$ and $e_1(I)$ of $I.$ We have  chosen  two 
theorems of Huckaba and Marley to 
illustrate the techniques and results in this area. Moreover  these theorems  
quickly yield, as special cases, several results proved over a period of 
fifty  years.     We will provide  simple proofs of these results and their 
consequences.     

In section two, we  will survey the main results proved  about the relationship of depth of $G(I)$ and Hilbert coefficients. In section three, we will provide a 
quick introduction to the theory of reductions of ideals. In section four, we prove the main facts for Hilbert polynomial in a one dimensional Cohen-Macaulay local ring. In section five, a detailed treatment of the theory of superficial elements and sequences 
is given. This theory is not available in the modern text books in commutative algebra. We have gathered the results useful for our purposes from the Chicago
Notes of M. P. Murthy \cite{m}, the recent book of C. Huneke and I. Swanson 
\cite{hs}  and several papers in this area. 
In section six, we  provide a new and very simple proof, using reductions
and superficial elements, of two theorems of Huckaba \cite{huc} and Marley 
\cite{hm}. We also include a few of their consequences.

\noindent
For undefined terms in this exposition we refer the reader to \cite{bh}.
\bigskip 

\noindent
{\bf Acknowledgements:} These notes are an expanded version of lectures delivered in October 2006 at University of Essen (Germany) while the author was working on a research project  funded by  {\em Deutsche Forschungsgemeinschaft} (DFG). He thanks J\"urgen Herzog for the invitation to visit Essen.  The manuscript was written at IIT Madras in July 2007 when the author was invited to deliver these lectures in the {\em Workshop in Commutative Algebra and Algebraic Geometry} funded by the {\em National Board for Higher Mathematics}. The author thanks A. V. Jayanthan and V. Uma for the invitation. Thanks are also due to S. A. Katre and Mousumi Mandal for their help. Finally I thank M. Rossi for several useful comments.

\section{\bf A brief survey}
In this section we  survey  results that link 
the depth of the associated graded ring $G(I)$ of  an $\mm$-primary ideal $I$ of a Cohen-Macaulay local ring $(R,\mm),$ and  linear relations among the Hilbert coefficients $e_0(I)$ and $e_1(I)$ of $I.$ We refer the reader to an excellent survey article by G. Valla \cite{v}  on Hilbert functions of graded algebras and in particular of a Cohen-Macaulay local ring. Perhaps 
the earliest known result in this direction is due to Northcott \cite{no}.

\begin{Theorem} [Northcott, 1960] Let $I$ be an $\mm$-primary ideal of a 
Cohen-Macaulay local ring $(R,\mm)$ with $R/ \mm $ infinite. Then \\
{\rm (a)}~ $e_0(I)-e_1(I) \leq \lm(R/I).$ \\
{\rm (b)}~  $e_1(I) \geq 0$ and equality holds   
if and only if $I$ is generated by $d$ elements. 
In this case $e_i(I)=0$ for $i=1, 2, \ldots,d$ and $G(I)$ 
is isomorphic to a polynomial ring  in $d$ indeterminates over $R/I.$   

\end{Theorem}

As a consequence of Northcott's theorem, we observe that a Cohen-Macaulay
local ring $(R,\mm)$ is regular if and only if $e_1(\mm)=0.$ The following 
theorem of Nagata \cite[(40.6)]{na} shows that regularity 
of $(R,\mm)$ can also be characterized in terms  of $e_0(\mm).$ 
Recall that a local ring $(R,\mm)$ is called unmixed  if for each associated
prime $p$ of the $\mm$-adic completion $\widehat{R}$ satisfies 
$\dim \widehat{R}/p= \dim R.$
\begin{Theorem}[Nagata, 1956] Let $(R,\mm)$ be an unmixed local ring. Then 
$e_0(I)=1$ if and only if $I=\mm$ and $R$ is regular.
\end{Theorem}  

 Huneke \cite{hun}   and Ooishi \cite{o} found conditions under which the equality  $e_0(I)-e_1(I)=\lm(R/I)$ holds.

\begin{Theorem}[Huneke, Ooishi, 1987] Let $(R,\mm)$ be a Cohen-Macaulay local ring
with $R/\mm$ infinite. Then $e_0(I)-e_1(I)=\lm(R/I)$ if and only if for 
any minimal reduction $J$ of $I,$ $JI=I^2.$ Moreover, when this is the case,
$G(I)$ is Cohen-Macaulay, $e_i(I)=0$ for all $i=2,3, \ldots,d$ and for all 
$n\geq 0,$
$$ H_I(n)=P_I(n)=e_0(I)\binom{n+d-1}{d}-e_1(I)\binom{n+d-2}{d-1}.$$
\end{Theorem}

\noindent
Since $e_0(I)=\lm(R/J)$ for any minimal reduction $J$ of $I,$ we  can restate
the Huneke-Ooishi theorem as $e_1(I)=\lm(I/J)$ if and only if $JI=I^2.$
Huckaba \cite{huc} and Huckaba-Marley \cite{hm} obtained interesting generalization of Huneke-Ooishi theorem.

\begin{Theorem}[Huckaba, 1996] Let $(R,\mm)$ be a Cohen-Macaulay local ring of dimension $d$ with infinite residue field. Let $J$ 
be a minimal reduction of $I.$  Then
$e_1(I) \leq \sum_{n \geq 1}\lm(I^n/JI^{n-1})$ and equality holds 
if and only if  depth $G(I) \geq d-1.$
\end{Theorem}

The Cohen-Macaulay property  of $G(I)$ was  characterized in terms  
of $e_1(I)$ by Huckaba and Marley \cite{hm}.

\begin{Theorem} [Huckaba-Marley, 1997] 
Let $(R,\mm)$ be a Cohen-Macaulay local ring of dimension $d$ with infinite 
residue field. Let $J$ 
be a minimal reduction of an $\mm$-primary ideal  $I.$  Then
$e_1(I) \geq \sum_{n \geq 1}\ell(I^n/J \cap I^{n})$ and equality holds 
if and only if $G(I)$  is Cohen-Macaulay.
\end{Theorem}

 We will provide simple proofs of both these theorems using induction on $d.$
One of the crucial tools used in the proofs is the so called 
{\em Sally Machine.} We will also provide a new proof of Sally machine due to 
B. Singh.

Now we turn to another line of research that relates depth of $G(I)$ with 
Hilbert coefficients.  The starting point is an inequality due to Abhyankar
\cite{a}. Let $\mu(I)$ denote the minimum number of elements required to generate 
an ideal $I$ in a local ring $(R,\mm).$

\begin{Theorem}[Abhyankar, 1967] Let $(R,\mm)$ be a Cohen-Macaulay local ring of dimension  $d.$ Then  
\begin{eqnarray} \label{abin}
e_0(\mm) & \geq & \mu(\mm)-d +1.
\end{eqnarray}   
\end{Theorem}  

J. Sally, in a long series of papers, investigated the effect of 
similar inequalities and equalities on the depth of $G(I).$ 
First she considered  Cohen-Macaulay rings in which (\ref{abin}) is an equality \cite{s1}. Such  rings are said to have {\em minimal multiplicity} or {\em maximal embedding dimension}.
 
\begin{Theorem}[Sally, 1977] Let $(R,\mm)$ be a $d$-dimensional Cohen-Macaulay local ring  with infinite residue field. Let $J$ be a minimal reduction of $\mm.$ Then
$R$ has minimal multiplicity if and only if $J\mm=\mm^2.$  In this case, 
$G(\mm)$ is  Cohen-Macaulay  and for all $n \geq 0,$
$$H_{\mm}(n)=P_{\mm}(n)= e_0(\mm)\binom{n+d-1}{d}-e_1(\mm)\binom{n+d-2}{d-1}.$$  
\end{Theorem}

\bigskip
\noindent {\bf Sally's conjecture}
\bigskip

We say that a Cohen-Macaulay local ring $(R,\mm)$ 
has {\em almost maximal embedding dimension}  or {\em almost minimal multiplicity } 
if  $\mu(\mm)=e_0(\mm)+d-2.$ Such rings have been a subject of investigation
since the appearance the paper \cite{s2} of J. D. Sally. In this paper, 
among other things, she proved the following

\begin{Theorem}[Sally, 1980]  
Let $(R,\mm)$ be a   Gorenstein local ring
of positive dimension $d$ and having almost maximal embedding dimension. 
Then $G(\mm)$ is Gorenstein and for all $n \geq 2,$
$$\lambda (\mm^n/\mm^{n+1})=e_0(\mm)\binom{n+d-2}{d-1}+\binom{n+d-3}{n}.$$ 
\end{Theorem}

In the paper \cite{s3}, Sally studied depth of $G(\mm)$  for  Cohen-Macaulay
local ring $(R,\mm)$  of almost  maximal dimension by means of their type.
Recall that the type of a Cohen-Macaulay local ring $R$, denoted by $\mbox{type}(R),$  is defined to be 
$\dim_{k}\Ext^d(k,R).$ For  rings of almost maximal dimension, depth of $G(\mm)$ is dependent on  $\mbox{type}(R).$    

\begin{Theorem} [Sally, 1983] Let $(R,\mm)$ be a Cohen-Macaulay local ring of 
positive dimension $d.$ Let $R$ have almost maximal embedding dimension. Then 
$\mbox{type}(R) \leq e_0(\mm)-2.$ If $\mbox{type}(R) < e_0(\mm)-2,$ then
$G(\mm)$ is Cohen-Macaulay, $R$ and $G(\mm)$ have the same type and for all 
$n \geq 2,$
$$\lm(\mm^n/\mm^{n+1})=e_0(\mm)\binom{n+d-2}{d-1}+\binom{n+d-3}{n}.$$  
\end{Theorem}

In \cite{s3}, Sally raised the question about depth of $G(\mm)$ for a 
Cohen-Macaulay local ring of almost maximal embedding dimension and maximal type 
$e_0(\mm)-2.$  This question 
remained open  for several years. It was answered independenly by M. E. Rossi 
and G. Valla in \cite {rv} and H-J Wang in \cite {w1}. The following theorem
summarizes the main results found in \cite{rv}.

\begin{Theorem}[Rossi-Valla, 1996] 
Let $(R,\mm)$ be a Cohen-Macaulay local ring of positive
dimension $d.$ Put 
$$G=G(\mm), \;\;H_G(n)=\lm(\mm^n/\mm^{n+1}),\; \mbox{and}\;\;   
P_G(z)=\sum_{n=0}^{\infty}H_G(n)z^n.$$ 
Then the following are equivalent:\\
{\rm (1)}~$R$ has almost maximal embedding dimension. \\
{\rm (2)}~ There is an integer $s$ such that $2 \leq s \leq  \mu(\mm)-d+1,$  and 
$$P_G(z)=\frac{1+(\mu(\mm)-d)z+z^s}{(1-z)^d}.$$
{\rm (3)}~If either of the above conditions  holds then $\depth G(I)  \geq d-1$ and $G$ 
is Cohen-Macaulay if and only of $s=2.$ 
\end{Theorem}

Rossi \cite{r} extended the above theorem partially to all $\mm$-primary ideals.
The condition $e_0(\mm)=\mu(\mm)-d+2$ has an analogue for $\mm$-primary ideals.
It is easy to see that for an $\mm$-primary ideal $I$ with a minimal reduction $J,$ in a Cohen-Macaulay local 
ring $(R,\mm)$ with infinite residue field,  $\lm(I^2/JI)=1$ if and only if $e_0(I)=\lm (I/I^2)+(1-d)\lm(R/I)+1.$ Rossi proved that If $\lm(I^2/JI)=1$ then
$G(I) \geq d-1$ \cite{r}.

In \cite{e}, Elias generalized the condition $\lm(I^2/JI)=1$ further 
and presented a unified  treatment of several  theorems.  Let $J$ be a 
minimal reduction of an $\mm$-primary ideal $I$ in $R.$ We say that $I$ and $J$ satisfy the $nth$ Valabrega-Valla condition $VV_n$ if $J\cap I^n=JI^{n-1}.$

\begin{Theorem}[Elias, 1999] Let $R$ be a Cohen-Macaulay local ring of dimension $d \geq 1.$
Let $I$ be an $\mm$-primary ideal of $R$ and $J$ be a minimal reduction of $I.$
Let $t$ be positive integer such that \\
$(1)$~~  $I$ and $J$ satisfy $VV_n$ for $n=0,1,\ldots,t,$ \\
$(2)$~~  $\lm (I^{t+1}/JI^t)=\delta \leq \min\{1,d-1\}.$\\
Then $d-\delta \leq \depth G(I)\leq d.$ If\; $t \geq e_0(I)-1,$ then $G(I)$ is Cohen-Macaulay. 
\end{Theorem}
\medskip
\noindent {\bf Sally Modules and depth of $G(I)$}
\medskip

Finally, we discuss the important notion of Sally modules introduced by Vasconcelos in \cite{vas}. Let $R$ be a Noetherian ring, and let $I$ be an ideal with a reduction $J.$ The {\em Rees algebra of an ideal $I$}, $\Rees(I),$ is defined to be the graded 
$R$-algebra $\bigoplus_{n=0}^{\infty} I^nt^n$ where $t$ is an indeterminate.  The Sally module $S_J(I)$ 
of $I$ with respect to $J$ is the $\Rees(J)$-module defined in the  exact sequence 
$$
0 \longrightarrow I\Rees(J) \longrightarrow I \Rees(I) \longrightarrow S_J(I):=\bigoplus _{n=0}^{\infty} I^{n+1}/IJ^n \longrightarrow 0.$$ 

We summarize some basic properties of Sally modules found in \cite{vas}.

\begin{Theorem}
Let $(R,\mm)$ be a $d$-dimensional Cohen-Macaulay local ring with infinite residue field. Let $I$ be and $\mm$-primary ideal and let  $J$ be a minimal reduction of 
$I.$  Then \\
$(1)$ If $S_J(I)\neq 0 $ then its  associated primes have height 1. In particular, the its dimension as an $\Rees(J)$-module is $d.$ \\
$(2)$ Let $S=S_J(I)=\bigoplus_{n=0}^{\infty}S_n.$ Then For large $n,$
$$H_I(n)=e_0(I)\binom{n+d-1}{d}+(\lm(R/I)-e_0(I))\binom{n+d-2}{d-1}-\lm(S_{n-1}).$$
Hence if $S\not=0$ then $\lm(S_n)$ is a polynomial function of degree $d-1.$ Let
$s_i$ for $i=0,1, \ldots,d-1$ be the Hilbert coefficients of $S.$ Then
$e_1(I)=e_0(I)-\lm(R/I)+s_0$ and for $i \geq 1,$ $e_{i+1}(I)=s_i.$ 
\end{Theorem}

Vaz Pinto studied the relationship of  the Sally module with the depth  
of $G(I).$ She proved the following interesting result:

\begin{Theorem} Let $(R,\mm)$ be a Cohen-Macaulay local ring with infinite residue field having  positive  dimension $d.$ Let $J$ be a minimal reduction of an $\mm$-primary ideal $I.$ Then the following conditions are equivalent:\\
$(1)$\;\; $s_0=\sum_{n=1}^{r_J(I)}\lm(I^{n+1}/JI^n).$\\
$(2)$\;\; $S$ is Cohen-Macaulay.\\
$(3)$\;\; $\depth G(I) \geq d-1.$
\end{Theorem}

\section{\bf Reductions of ideals}

\begin{Definition}
 Let $R$ be a Noetherian ring, $J\subseteq I$ be ideals of $R.$ If $JI^n=I^{n+1}$ for some $n$  then $J$ is called a reduction of $I$. The reduction number  $r_J(I)$ of $I$ with respect to $J$ is the smallest $n$  such that $JI^n=I^{n+1}$.
The reduction number $r(I)$ of $I$ is the smallest among the reduction numbers 
$r_J(I)$ where $J$ varies over all minimal reductions of $I.$
\end{Definition}

The notion of reduction of an ideal was introduced by Northcott and Rees in the paper \cite{nr}. This paper, now a classic, introduced several other important concepts such as fiber 
cone of an ideal, analytic spread, analytically independent elements etc. Reductions
have  played a crucial role in understanding Hilbert coefficients, Rees algebras, fiber cones and associated graded rings of ideals. In this section we prove their basic properties to be used in the later sections. 
\begin{Proposition}
 Let $J\subseteq I$ be ideals of a Noetherian ring $R$. Then $J$ is reduction of $I$ if and only if $R[It]=\bigoplus_{n=0}^{\infty} I^nt^n$ is a finite $R[Jt]$-module.
\end{Proposition}
\begin{proof} 
Let $J$ be a reduction of $I$. Then
 $(J^kt^k)(I^nt^n)=I^{n+k}t^{n+k}$ for all $k \geq 1$ and for $n \geq r$ for some $r$. Therefore 
\begin{eqnarray}
\label{2.2}
R[It]= R[Jt]+R[Jt](It)+\cdots+R[Jt](I^r t^r). 
\end{eqnarray}

Hence $R[It]$ is a finite $R[Jt]$-module. Conversely, let $R[It]$ be a finite $R[Jt]$-module. Since $R[It]$ is a graded $R[Jt]$-module, there is a  finite set of homogeneus generators of $R[It]$ as an $R[Jt]$-module. Thus (\ref{2.2}) follows for some $r.$  Equate the components of degree $r+1$ on both sides of (\ref{2.2}) to get
$$
I^{r+1}t^{r+1}=J^{r+1}t^{r+1}+J^rIt^{r+1}+\cdots+JI^rt^{r+1}.
$$
Thus $JI^r=I^{r+1}$. Hence $J$ is a reduction of $I.$
\end{proof}
\begin{Corollary}
Let $K \subseteq J \subseteq I$ be ideals of a Noetherian ring $R. $ Then $K$ is a reduction of $I$ if and only if $K$ is a reduction of $J$ and $J$ is a reduction of $I$.
\end{Corollary}
\begin {proof} Let $K$ be a reduction of $I$. Then $R[It]$ is a  finite $R[Kt]$-module. Hence $R[Jt]$ is a finite $R[Kt]$-module and $R[It]$ is a finite  $R[Jt]$-module. Hence $K$ is a reduction of $J$ and $J$  is a reduction of $I$. 
Conversely let $K$ be a reduction of $J$ and $J$ be a reduction $I.$ Then $R[Jt]$ is a finite $R[Kt]$-module and $R[It]$ is a finite $R[Jt]$-module. Hence $R[It]$ is a finite $R[Kt]$-module. Thus $K$ is a reduction of $I$.
\end{proof}
\begin{Proposition}
 Let $(R,\mm)$ be a local ring and let  $I$ be an ideal of  $R$. Then an ideal $J \subseteq I$ is a reduction of $I$ if and only if $J +\mm I$ is a reduction of $I$.
\end{Proposition}
\begin{proof}
 As $J \subseteq J+\mm I \subseteq I,~ J+\mm I$ is a reduction of $I$. Conversely, let $(J+\mm I)I^n=I^{n+1}.$ Then $JI^n+\mm I^{n+1}=I^{n+1}.$ By Nakayama's Lemma $JI^n=I^{n+1}$. Hence $J$ is a reduction of $I.$
\end{proof}

\begin{Definition}Let $I$ be an ideal of a local ring $(R,\mm)$. The fiber cone $F(I)$ of $I$ is the graded algebra $R[It]/\mm R[It]=\bigoplus_{n=0}^{\infty}I^n/{\mm I^n}$. The Krull dimension of $F(I),$ denoted by $\ell(I),$ is called the analytic spread of $I$. An ideal $J\subseteq I$ is called a  minimal reduction of $I$ if 
$J' \subseteq J$ and $J'$ is a reduction of $I$ then $J=J'$.
\end{Definition} 

We will prove that minimal reductions exist and if $R/\mm$ is infinite then all minimal reductions of an ideal $I$  require exactly $\ell(I)$ minimal generators.
 
\vspace{0.3cm}

\begin{Proposition} Let $(R,\mm)$ be a local ring, $I$ an ideal of $R$. For $a \in I,$ put $a^0=a+\mm I \in I/{\mm I}$. Let $J=(a_1,\,a_2,\,\ldots,\,a_s)\subseteq I$. Then $J$ is a reduction of $I$ if and only if $(a_1^0,\,a_2^0,\,\ldots,\,a_s^0)$ is primary  for the maximal homogeneous ideal $F(I)_+.$ In particular $\mu(J)\geq \ell(I)$.
\end{Proposition}
\begin{proof} Let $J$ be a reduction of $I$. Then there is an $r$ such that $JI^r=I^{r+1}$. Note that  $(a_1^0,\,a_2^0,\,\ldots,\,a_s^0)_n=JI^{n-1}+\mm I^n/\mm I^n$ for $n\geq1$. Hence $(a_1^0,\,a_2^0,\,\ldots,\,a_s^0)_n=F(I)_n$ for $n\geq r+1$. Hence  $(a_1^0,\,a_2^0,\,\ldots,\,a_s^0)$ is $F(I)_+$-primary. 

Conversely let $(a_1^0,\,a_2^0,\,\ldots,\,a_s^0)$ be $F(I)_+$-primary. Hence
there exists $r$ so that for $n \geq r,$ $F(I)_n=(a_1^0,\,a_2^0,\,\ldots,\,a_s^0)_n.$  
Thus $JI^{r-1}+\mm I^r=I^{r}$. By Nakayama's Lemma, $JI^{r-1}=I^r$. Hence $J$ is a reduction of $I$. By Dimension Theorem $\mu(J)\geq \ell(I)$.
\end{proof}

\begin{Proposition} Let $J\subseteq I$ be a reduction of an ideal $I$ in a local ring $(R,\mm)$. Then $J$ contains a minimal reduction of $I$. Let $a_1,\,a_2,\,\ldots,\,a_s\in J$ be such that $a_1^0,\,a_2^0,\,\ldots,\,a_s^0\in I/{I\mm}$ are linearly independent and $s$ is minimal with respect to the property that $K=(a_1,\,a_2,\,\ldots,\,a_s)$ is a reduction of $I$ contained in $J$.
Then $K$ is a minimal reduction of $I$ contained in $J.$
\end{Proposition}

\begin{proof} Suppose $K' \subseteq K$ and $K'$ be a reduction of $I$. Let $f\,:\,K/{\mm K}\longrightarrow I/{\mm I}$ be the natural map of $k:=R/\mm$-vector spaces. Since $a_1^0,\,a_2^0,\,\ldots,\,a_s^0 \in I/\mm I$ are $k$-linearly independent, $a_1+\mm K,\cdots,a_s+\mm K$ are $k$-linearly independent in $K/{\mm K}$. Hence $\Ker f= K\cap \mm I/\mm K=0$. Therefore $ K\cap \mm I=\mm K$. 

Next observe that $K'+\mm I=K+\mm I$. Indeed let $K'+\mm I< K+\mm I$. Then $K'+\mm I/\mm I $ is a proper subspace of $K+\mm I/\mm I$. Let 
$t=\dim (K'+\mm I)/\mm I$ and $b_1,\,b_2,\,\ldots,\,b_t\in K$ such that $b_1^0,\,b_2^0,\,\ldots,\,b_t^0\in I/{I\mm }$ are linearly independent. Since $K'$ is a reduction of $I$, $\dim F(I)/{(b_1^0,\,b_2^0,\,\ldots,\,b_t^0)}=0$. This contradicts the minimality of $s$.

Thus $K\subseteq (K'+\mm I)\cap K=K'+(\mm I\cap K)=K'+\mm K$. By Nakayama's Lemma $K'=K$. Therefore $K$ is a minimal reduction of $I.$
\end{proof}

\begin{Proposition}
Let $(R, \mm)$  be a local ring with infinite residue field $k.$ 
Let $I$ be an ideal of $R$ and $ a_1, a_2, \cdots, a_s \in I.$ Then 
$J = (a_1, a_2, \ldots, a_s)$ is a minimal reduction of $I$ if and only if $(a_1^{0}, a_2^{0}, \ldots, a_s^{0})$ is a homogeneous system of parameters of $F(I)$.
\end{Proposition}

\begin{proof} Let $J$ be a minimal reduction of $I$. Put $\ell=\ell(I).$ Then $s$ is smallest with respect to the property that $\dim F(I)/ (a_1^{0}, a_2^{0}, \cdots, a_s^{0}) =0$. Let $s > \ell(I)$. Since $k$ is infinite, by Noether Normalization, there exist $b_1, b_2, \cdots, b_{\ell}\in I$ such that $F(I)$ is integral over the polynomial ring $k[b_1^{0}, b_2^{0}, \ldots, b_{\ell}^{0}]$. Hence $(b_1^{0}, b_2^{0}, \cdots, b_{\ell}^{0})F(I)$ is zero-dimensional. Therefore $(b_1, b_2, \cdots, b_{\ell})$ is a reduction of $I$. This contradicts minimality of $s$. Hence $s = \ell(I)$.

Conversely let $a_1 ,a_2 \cdots, a_s \in I$ such that $a_1^{0}, \ldots, a_s^{0}$ is a homogeneous system of parameters of $F(I)$. Then $(a_1 ,a_2,\cdots,a_s) = J$ is a reduction of $I$. By the above proposition it is a minimal reduction of $I$.
\end{proof}

\noindent Recall that {\em altitude } of an ideal $I,$ $\mbox{alt}~ I,$ is the maximum of the 
heights  of the minimal primes over $I.$

\begin{Corollary} Let   $I$ be an ideal of a local ring $(R,\mm).$ Then 
$$\mbox{alt} \leq  \ell(I) \leq \dim R.$$
\end{Corollary}

\begin{proof}  We may assume that $R/\mm $ is infinite. 
Let $J$ be a minimal reduction of $I$. Then $V(I) = V(J)$. If $p$ is any minimal prime of $J.$ Then $\hei p \leq \mu(J) = \ell(I)$. Since $\lambda(I^{n} /\mm I^{n}) \leq \lambda(I^{n} /I^{n + 1})$ for all $n,$  the degree of Hilbert polynomial of $F(I)$ is atmost that of the Hilbert polynomial of $G(I)$. Hence $\ell(I)=\dim F(I) \leq \dim G(I)$ = $\dim R$. 
\end{proof}

\section{\bf Hilbert function in $1$-dimensional Cohen-Macaulay local rings}
Throughout this section $(A,\mm)$ is a $1$-dimensional Cohen-Macaulay local ring with $A/\mm$ infinite.
 Let $I$ be an $\mm-$primary ideal. The Hilbert function of $I$ is the function $H_I(n)=\lambda (A/I^n).$ 
The Hilbert polynomial of $I$, $P_I(n)$ has degree $1$. Write $$P_I(n)=e_0n-e_1.$$
The {\em postulation number} of $I$ is defined to be 
$$n(I)=\max \{n ~|~ H_I(n)\not=P_I(n)\}.$$
Since $\hei I=1=\dim A$, $\ell(I)=1$. Since $A/\mm$ is infinite, there exists 
$a \in I$ such that $(a)$ is a reduction of $I.$ 
\begin{Theorem}[Northcott, 1960] Let $I$ be an $\mm$-primary ideal of a one dimensional Cohen-Macaulay local ring $(A,\mm).$  Let $(a)$ be a minimal reduction of $I.$ Then\\
\noindent
$(1)~ P_I(n+1)-H_I(n+1)\geq P_I(n)-H_I(n)$ for all $n\geq 0.$ \\
$(2)~ e_0-e_1 \leq \lambda (A/I).$\\
$(3)~ e_1\geq 0$ with equality if and only if $I$ is principal.
\end{Theorem}
\begin{proof} 
Notice that for all $n\geq 0,$
\begin{eqnarray} \nonumber
P_I(n+1)-H_I(n+1) & = & (n+1)e_0-e_1-\lambda (A/I^{n+1})\\ 
\nonumber &= & ne_0-e_1-\lambda ((a)/aI^n)+\lambda (I^{n+1}/aI^n)\\ 
 \label{3.1} &= & P_I(n)-H_I(n)+\lm (I^{n+1}/aI^n).
\end{eqnarray}
Hence $H_I(n)-P_I(n)\geq H_I(n+1)-P_I(n+1).$ 
For large $n$, $H_I(n)=P_I(n)$, hence for all $n\geq 0,$ $H_I(n)\geq P_I(n)$. For $n=1$ we get $\lm (A/I)\geq e_0-e_1$. Thus $e_1\geq e_0-\lm (A/I)=\lm (A/(a))-\lm (A/I)=\lm (I/(a))\geq 0$.
Therefore if  $e_1=0$ then $I=(a)$. Conversely if $I=(a)$ then for all $n \geq 1,$
$$\lm (A/(a)^n) = \lm (A/(a))+\lm ((a)/(a)^2)+\cdots +\lm ((a)^{n-1}/(a)^n)
=  ne_0.$$
Hence $e_1(a)=0.$
\end{proof}
\begin{Proposition} \label{rednum}Let $(A,\mm)$ be a one dimensional Cohen-Macaulay local ring.
Let $(a)$ be a minimal reduction of an $\mm$-primary ideal $I.$ Then
$$r_{(a)}I=n(I)+1.$$
\end{Proposition}
\begin{proof}
By (\ref{3.1}), for all $n \geq 0,$ we have 
$$H_I(n+1)-P_I(n+1)=H_I(n)-P_I(n)-\lm (I^{n+1}/aI^n).$$
Put $k=n(I)$ and $r=r_{(a)}I$. Hence for all $n\geq r,$~~
$H_I(n)-P_I(n)=H_I(r)-P_I(r).$
But $H_I(n)=P_I(n)$ for large $n.$  Hence $H_I(r)=P_I(r)$. Thus $k\leq r-1$. To prove $k\geq r-1$, put $n=k+1$ in (\ref{3.1}) to get $$H_I(k+2)-P_I(k+2)=H_I(k+1)-P_I(k+1)-\lm (I^{k+2}/aI^{k+1}).$$
Hence $\lm (I^{k+2}/aI^{k+1})=0$. Thus $I^{k+2}=aI^{k+1}$ and hence $r\leq k+1$.
\end{proof}
\begin{Proposition} Let $(A,\mm)$ be a one dimensional Cohen-Macaulay local ring.
Let $(a)$ be a minimal reduction of an $\mm$-primary ideal $I.$ Then
$e_0-e_1=\lm (A/I)$ if and only if $aI=I^2$.
\end{Proposition}
\begin{proof}
Let $aI=I^2$. Then $n(I)\leq 0$. Hence $H_I(1)=P_I(1)$. This gives $e_0-e_1=\lm (A/I)$. Conversely let $e_0-e_1=\lm(A/I).$ By (\ref{3.1}), for $n \geq 1,$ 
$$0 \leq H_I(n)-P_I(n)\leq \lm (A/I)-e_0+e_1=0.$$  
By Proposition \ref{rednum}, $n(I)=r_{(a)}-1 \leq  0.$ Hence $(a)I=I^2.$

\end{proof}
\begin{Theorem}[Huckaba-Marley] \label{dim1} 
Let $(A,\mm)$ be a one dimensional Cohen-Macaulay local ring. Let $J:=(a)$ be a reduction of an $\mm$-primary ideal $I.$
Then \\
$(1)~ e_1(I)=\sum_{n\geq 1}\lm (I^n/JI^{n-1})\geq \sum_{n\geq 1}\lm (I^n/J\cap I^n).$\\
$(2)~ e_1(I)=\sum_{n\geq 1}\lm (I^n/J\cap I^n)$ if and only if $G(I)$ is Cohen-Macaulay.
\end{Theorem}
\begin{proof}(Rossi)
For all $m\geq 1$ we have
$$\lm(I^{m-1}/I^m)=\lm(A/J)-\lm(I^m/JI^{m-1})=e_0(I)- \lm(I^m/JI^{m-1}).$$
Adding the above equation for $m=1,2,\ldots,n$ we obtain
$$\lm(A/I^n)=ne_0(I)-\sum_{m=1}^n\lm (I^m/JI^{m-1}).$$
Taking $n$ large in the above equation we get,  
$$P_I(n)= ne_0(I)-e_1(I)=ne_0(I)-\sum_{m=1}^{r_J(I)} \lm(I^m/JI^{m-1}).$$
Thus $e_1(I)=\sum_{m=1}^{r_J(I)} \lm(I^m/JI^{m-1}).$ Hence
$e_1(I) \geq \sum_{m=1}^{r_J(I)} \lm(I^m/J\cap I^{m})$   and equality holds if and only if $I^m\cap J=JI^{m-1}$  for all $m \geq 1.$ The last condition is equivalent to
$G(I)$ being Cohen-Macaulay.
\end{proof}
\begin{Example}{\rm
Let $k$ be a field and $t$ be an indeterminate. Let $A=k[[t^3,t^4,t^5]]$ and $\mm=(t^3,t^4,t^5)$. Check that $t^3\mm=\mm^2$. Since $A$ is Cohen-Macaulay, $e_0(\mm)=e_0(t^3)=\lm (A/t^3A)=3.$
Since $t^3\mm=\mm^2$ and $e_0-e_1=\lm(A/\mm)=1,$ we get $e_1(\mm)=2$. Therefore $P_\mm(n)=3n-2$ for all $n\geq 1.$}
\end{Example}

\section{\bf Superficial sequences}
In this section we develop the theory of superficial elements and superficial sequences. Existence of a  superficial element in an ideal $I$ allows us to relate Hilbert coefficients of $I$ and those of $I/(a).$ As a result we can first investigate Hilbert coefficients in dimension one and lift this information to higher dimension.
\begin{Definition}
Let $I$ be an ideal of a local ring $(A,\mm)$. We say $a\in I$ is \textit{superficial} if there is a $c\geq 0$ such that $(I^n \colon a)\cap I^c=I^{n-1}$ for all $n>c$.
\end{Definition}
\begin{Proposition} Let $I$ be an ideal of a local  ring $(A,\mm).$\\
$(1)$ If $I$ is nilpotent then every $a\in I$ is superficial for $I$.\\
$(2)$ If $I$ is not nilpotent then a superficial element $a$ of $I$ satisfies $a\in I\smallsetminus I^2$.

\end{Proposition}
\begin{proof}
(1)  Let $I^r=0$. Then for $c=r+1$ and $n>r+1,$ we have $(I^n:a)\cap I^c=I^{n-1}=0$ for any $a\in I$. Hence $a$ is superficial for $I.$ \\
(2) Suppose  $I$  is  not nilpotent. Suppose $a$ is superficial for $I$ and for all $n>c$, $(I^n:a)\cap I^c=I^{n-1}$. Suppose $a\in I^2$. Put $n=c+2$ to get $(I^{c+2}:a)\cap I^c=I^{c+1}$. As $aI^c \subseteq I^{c+2},$  $I^c=I^{c+1}$. By Nakayama's Lemma $I^c=0$. This is a contradiction. Hence $a\in I\smallsetminus I^2$.
\end{proof}
\begin{Proposition}
Let $(A,\mm)$ be a local ring and let  $I$ be an ideal. Let $a\in I\smallsetminus I^2$ and $a^*=a+I^2$. Then $a$ is superficial for $I$ if and only if the multiplication map $a^*:I^n/I^{n+1}\longrightarrow I^{n+1}/I^{n+2}$ is injective for large $n.$
\end{Proposition}
\begin{proof}
Let $(I^n:a)\cap I^c=I^{n-1}$ for all $n>c$. Suppose $n>c$, $b\in I^n$ and $b^*a^*=0$. Then $ba \in I^{n+2}$. Therefore $b\in (I^{n+2}:a)\cap I^c=I^{n+1}$. So $b^*=0$. Hence the map $a^*$ is injective for large $n.$

Conversely let $a^* : I^n/I^{n+1}\longrightarrow I^{n+1}/I^{n+2}$ be injective for $n > c.$ We show $(I^n:a)\cap I^c=I^{n-1}$ for all $n>c$. Let 
$b\in (I^{n}:a)\cap I^c$. Let $b\in I^m \smallsetminus I^{m+1}$. Then $b^*a^*\in I^{m+1}/I^{m+2}$. If $b^*a^*=0$ then $ab \in I^{m+2}$. Thus $b\in (I^{m+2}:a)\cap I^c=I^{m+1}$ which is a  contradiction. Therefore $b^*a^*\not= 0$ and consequently $ab \notin I^{m+2}$ and $ab \in I^n$. Thus $n<m+2$. Hence $b\in I^m\subseteq I^{n-1}$ and  $(I^n :a)\cap I^c=I^{n-1}$.
\end{proof}
\let\ov=\overline
\noindent {\bf Existence of superficial elements}

\begin{Proposition}
Let $(A,\mm)$ be a local ring with $A/\mm$ infinite. Let $M$ be an $A$-module. If $N_1, N_2,\ldots, N_t$ are proper submodules of  $M$ then $$N_1\cup N_2\cup \ldots \cup N_t<M.$$
\end{Proposition}
\begin{proof}
We apply induction on $t.$ The $t=1,2$ cases  are trivial. Suppose that $t \geq 3$ and let
$M=N_1\cup N_2 \cup \cdots \cup N_t.$ We may assume that $N_1\nsubseteq N_2\cup N_3\cup \ldots \cup N_t$ and $N_2\cup N_3\cup \ldots \cup N_t\nsubseteq N_1$. As 
$A/\mm$ is infinite there are units $u_1,u_2,\ldots \in A $ such that $u_i-u_j$ is a unit for $i\not= j$ in $A$. Let $a \in  N_1\smallsetminus (N_2\cup N_3\cup \ldots \cup N_t)$ and $b\in (N_2\cup N_3\cup \ldots \cup N_t)\smallsetminus N_1$. Then $a+ub$ and $a+wb\in N_j$ for some $j$ and distinct units $u,w\in A$ by Pigeon-Hole Principle. Since 
$$(a+ub)-(a+wb)=(u-w)b\in N_j,$$ 
and $u-w$ is a unit, $ b \in N_j$. 
By choice of $b$, $j\not= 1$. Since $w-u$ is a unit and 
$$w(a+ub)-u(a+wb)=(w-u)a\in N_j,$$ we conclude 
$a\in N_j$. The choice of $a$ forces $j=1$. This is a contradiction. Thus $N_1\cup N_2\cup \ldots \cup N_t<M$.
\end{proof}
\begin{Theorem}
Let $(A,\mm)$ be a local ring with $A/\mm$ infinite. Let $I,J_1,\ldots ,J_t$ be $A-$ideals with $I\nsubseteq J_1\cup \ldots \cup J_t$. Then there exists $a\in I\smallsetminus (J_1\cup \ldots \cup J_t)$ such that $a$ is superficial for $I$.
\end{Theorem}
\begin{proof}
First note that the $A/I-$submodules $(J_i\cap I)+I^2/I^2$, $i=1,2,\ldots ,t$ are proper submodules of $I/I^2$. Indeed, let $(J_i\cap I)+I^2=I$. By Nakayama's Lemma $J_i\cap I=I$. Hence $I\subseteq J_i,$ which is a contradiction. Let 
$$(0)=Q_1\cap \ldots \cap Q_s \cap Q_{s+1}\cap \ldots \cap Q_g$$ be a reduced primary decomposition of $(0)$ in $G(I)$. Put $G_n=I^n/I^{n+1}.$ Let $\sqrt{Q_i}=P_i$ for $i=1,2,\ldots ,g$. Suppose $G_1\nsubseteq P_i$ for $i=1,\ldots ,s$ and $G_1\subseteq P_j$ for $j=s+1,\ldots ,g$. Therefore $G_1\cap P_1,\ldots, G_1\cap P_s$ are proper $G_0$-submodules of $G_1$. By previous proposition, there is an $a\in I\smallsetminus I^2$ such that 
$$a^*\in G_1 \smallsetminus\{\displaystyle{[\cup_{i=1}^s}P_i]\cup \displaystyle{[\cup_{i=1}^t}(J_i\cap I)+I^2/I^2]\}.$$ 
We claim that $a$ is superficial for $I$. For this it is enough to show that $(0:a^*)\cap G_n=0$ for large $n.$  Suppose $b^*a^*=0$. Since $a^*\notin P_i$ for $i=1,\ldots ,s$, $b^*\in (Q_1\cap \ldots \cap Q_s)$. Since $Q_j$ is $P_j-$primary for all $j$, $P_j^N\subseteq Q_j$ for $N$ large. Hence $G_1^N=I^N/I^{N+1}\subseteq Q_j$ for $j=s+1,\ldots ,g$. Therefore
$$G_N\cap (0:a^*)\subseteq Q_1\cap \ldots \cap Q_s \cap Q_{s+1}\ldots \cap Q_g=(0).$$ Hence $a$ is superficial for $I.$
\end{proof}

\medskip
\noindent {\bf Superficial sequences and reductions}
\medskip

\begin{Definition}
Let $(A,\mm)$ be a local ring, and let  $I$ be an $A$-ideal. A sequence $x_1, x_2, \ldots, x_s\in I$ is called a \textit{superficial sequence} for $I$ or $I$-superficial sequence if $\ov{x_i}$ is superficial for $I/(x_1,\ldots, x_{i-1})$ for $i=1,2,\ldots,s.$ 
\end{Definition}
\begin{Lemma}\label{sup}
Let $x_1,\ldots, x_s$ be an $I$-superficial sequence for $I$. Then for  
$n>>0$, $$I^n\cap (x_1,\ldots ,x_s)=(x_1,\ldots, x_s)I^{n-1}.$$
\end{Lemma}
\begin{proof}
Induct on $s$. Let $s=1$. As $x_1$ is superficial for $I$, there is $c\geq 0$ such that for $n >  c$,
$$(I^{n+1}:x_1)\cap I^c=I^n.$$
By Artin-Rees Lemma, there is a $p$ such that $I^n\cap x_1A=I^{n-p}(I^p\cap x_1A) \subseteq x_1I^{n-p}$ for all $n\geq p$. We now show $I^n\cap x_1A=x_1I^{n-1}$ for all $n\geq c+p$. Let $y=bx_1\in I^n\cap x_1A$ for some $b\in A$. Then $y\in I^{n-p}x_1$. Hence $y=dx_1$, where $d\in I^{n-p}\subseteq I^c$, since $n-p\geq c$. Therefore $(b-d)x_1=0$. Hence $b-d\in (0:x_1)\subseteq (I^n:x_1)$. Now $d=b-(b-d)\in (I^n:x_1)\cap I^c=I^{n-1}$. Hence $I^n\cap x_1A=x_1I^{n-1}$ for $n\geq p+c$. 
Now let $s\geq 2$. By induction, for large $n,$ 
$$I^n\cap (x_1,\ldots ,x_{s-1})=(x_1,\ldots x_{s-1})I^{n-1}.$$
As $\ov x_s$ is superficial for $A/(x_1,\ldots ,x_{s-1})$, by $s=1$ case, for large $n,$ 
$$(\ov x_s)\cap [I/(x_1,\ldots , x_{s-1})]^n=(\ov x_s)[I/(x_1,\ldots , x_{s-1})]^{n-1}.$$   Hence for large $n,$
\begin{eqnarray*}
x_sI^{n-1}+(x_1,\ldots , x_{s-1})&=&(x_1,\ldots ,x_s)\cap [I^n+(x_1,\ldots , x_{s-1})]\\
&=& I^n\cap (x_1,\ldots , x_s)+(x_1,\ldots , x_{s-1}).
\end{eqnarray*}
Therefore
\begin{eqnarray*}
I^n\cap (x_1,\ldots , x_s)& \subseteq & x_sI^{n-1}+ (x_1,\ldots , x_{s-1})\cap I^n\\
& =& x_sI^{n-1}+ (x_1,\ldots , x_{s-1}) I^{n-1}\\
& = &(x_1,\ldots , x_{s-1}) I^{n-1}.
\end{eqnarray*}
\end{proof}
\begin{Lemma}\label{4.8}
Let $J\subseteq I$ be ideals of a local ring $A.$ . Let $a\in J$ be superficial for $I$. If $J/(a)$ is a reduction of $I/(a)$ then $J$ is a reduction of $I$.
\end{Lemma}
\begin{proof}
Let $J/(a)$ be a  reduction of $I/(a).$ Then for large $n$, 
$JI^n+(a)=I^{n+1}+(a)$. Hence $I^{n+1}\subseteq JI^n+(a)\cap I^{n+1}=JI^n+aI^n=JI^n$ for large $n.$  Hence $JI^n=I^{n+1}$. Thus $J$ is a reduction of $I$.
\end{proof}

\begin{Proposition}
Let $I$ be an ideal of a local ring $(A,\mm)$. Let $a$ be superficial for $I$ and $a^*=a+I^2$. Then for large $n,$
$$[G(I)/a^*]_n\cong G(I/aA)_n.$$
\end{Proposition}
\begin{proof}
On one hand we have for large $n,$
$$G(I/aA)_n=\frac{I^n+aA}{I^{n+1}+aA}\cong \frac{I^n}{I^{n+1}+(aA\cap I^n)}
=\frac{I^n}{I^{n+1}+aI^{n-1}}.$$
On the other hand for large $n,$
$$[G(I)/a^*]_n=\frac{I^n/I^{n+1}}{aI^{n-1}+I^{n+1}/I^{n+1}}\cong \frac{I^n}{aI^{n-1}+I^{n+1}}.$$
\end{proof}
\begin{Corollary}
Let $(A,\mm)$ be a $d-$dimensional local ring and let  $I$ be an $\mm$-primary ideal. Suppose $a$ is superficial for $I$. Then $(a)$ is parameter, i.e. $\dim A/aA=d-1$. Moreover if $d=1$ then  $(a)$ is a reduction of $I$.
\end{Corollary}
\begin{proof}
By Lemma \ref{sup}, $I^n\cap aA=aI^{n-1}$ for large $n.$ 
Consider the exact sequence 
$$0\longrightarrow (0:a^*)_{n-1}\longrightarrow G(I)_{n-1}\xrightarrow{\; \; a^*} G(I)_n\longrightarrow [G(I)/a^*]_n\longrightarrow 0.$$
As $a$ is superficial for $I$, $(0:a^*)_n=0$ for large $n.$ Hence for large $n,$
$$\lambda (I^{n}/I^{n+1})- \lambda (I^{n-1}/I^n)=\lambda[G(I/a)]_n.$$ 
Thus $\lambda[G(I/a)]_n$ is a polynomial function of degree $d-2$. Hence $\dim A/aA=d-1$. \\
If $d=1$, then $A/aA$ is Artin. Hence  $aA$ is $\mm-$primary. Therefore $I^n\subseteq aA$ for large  $n.$. But $I^n \cap aA=aI^{n-1}$ for large $n.$
Hence for large $n,$ $I^n=aI^{n-1}.$ Thus $(a)$ is a reduction of $I.$
\end{proof}
\begin{Theorem}
Let $(A,\mm)$ be a $d-$dimensional local ring. Let $a_1,\ldots ,a_d$ be a superficial sequence for $I$. Then $J=(a_1,\ldots ,a_d)$ is a minimal reduction of $I$. 
\end{Theorem}
\begin{proof}
Apply induction on $d$. We have proved this for $d=1$. Since $\ov {a_2}, \ov {a_3}, \ldots ,\ov{a_d} $ is a superficial sequence for $I/(a_1)$ in the $(d-1)$-dimensional local ring $A/a_1A$, $\ov {a_2}, \ldots \ov {a_d}$ is a reduction of $I/(a_1)$. By Lemma \ref{4.8}, $J$ is a reduction of $I$. 
\end{proof}
\begin{Proposition}
Let $(A,\mm)$ be a $d-$dimensional local ring with $A/\mm$ is infinite. Let $J$ be a minimal reduction of an $\mm$-primary ideal $I.$ Then $J$ can be generated by a superficial sequence for $I$.
\end{Proposition}
\begin{proof}
Put $J=(b_1, b_2, \ldots ,b_d)$ and  $b_i^*=b+I^2$, for $i=1,2,\ldots, d$. Then $(b_1^*,\ldots ,b_d^*)_n=(JI^{n-1}+I^{n+1})/I^{n+1}$. Hence 
$$[G(I)/(b_1^*,\ldots ,b_d^*)]_n\cong I^n/(JI^{n-1}+I^{n+1})=0 
\;\mbox{ for}\; n\geq r_J(I)+1.$$ 
Thus $b_1^*,\ldots ,b_d^*$ is a homogeneous system of parameters for $G(I)$. Apply induction on $d.$  If $d=0$ then $(0)$ is a reduction of $I$. Let $d\geq 1$ and $P_1,\ldots ,P_s$ be the relevant associated primes of $G(I).$ In other words
$G_1 \nsubseteq P_i$ for $i=1,2,\ldots,s.$
 Then $\dim G(I)=d > \hei P_i $ for all $i$. Moreover $J+I^2/I^2\nsubseteq P_i$ for all $i$. By the  prime avoidance lemma for homogeneous ideals, 
there is an $a_1\in J\smallsetminus \mm J$ such that $a_1^*\notin \displaystyle {\cup_{i=1}^s}P_i$. Hence  $a_1$ is superficial for $I$. Since $J/(a_1)$ is a minimal reduction of $I/(a_1)$, by induction, there exists a superficial sequence $\ov{a_2},\ldots ,\ov{a_{d-1}}$ for $I/(a_1)$ such that $J/(a_1)=(\ov{a_2},\ldots ,\ov{a_{d-1}})$. Thus $J=(a_1,\ldots ,a_d)$.
\end{proof}

\medskip
\noindent {\bf Superficial elements and Hilbert polynomials}
\medskip

Let $I$ be an $\mm-$primary ideal of a $d$-dimensional local ring $(A,\mm)$ with $A/\mm$ infinite. The Hilbert polynomial of $I$, $P_I(n),$ is written as $$P_I(n)=e_0(I)\binom{n+d-1}{d}-e_1(I)\binom{n+d-2}{d-1}+\cdots +(-1)^d e_d(I). $$
We have seen that if $a$ is superficial for $I$ then $\dim A/aA=d-1$. We will now study the relationship between Hilbert coefficients of $I$ and those of $I/aA$. We will see that superficial sequences provide us with an inductive tool to study the Hilbert polynomial. Let $f:\mathbb{N}\longrightarrow \mathbb{N}$ be a function. Put $\triangle f(n)=f(n)-f(n-1)$.

\begin{Theorem}
Let $I$ be an $\mm-$primary ideal of a $d$-dimensional local ring $(A,\mm)$. Let $a$ be a superficial element for $I$. Let $\ov A=A/(a)$ and $\ov I=I/(a)$. Then \\
$(1)$  $P_{\ov I}(n)= \triangle P_I(n)+\lm (0:a)$. Hence $\dim A/(a)=d-1.$\\
$(2)$ For $i=0,1,\ldots ,d-2,$  $e_i(\ov I)=e_i(I)$ and $e_{d-1}(\ov I)=e_{d-1}(I)+\lm (0:a).$
\end{Theorem} 
\begin{proof}
By the exact sequence 
$$0\longrightarrow (I^n:a)/I^{n-1}\longrightarrow A/I^{n-1}\xrightarrow{\; \; a} A/I^n\longrightarrow A/(I^n,a)\longrightarrow 0$$
we get $\lm (A/(I^n,a))=\lm (A/I^n)-\lm(A/I^{n-1})+\lm ((I^n:a)/I^{n-1})$. Thus for large $n,$ 
$$P_{\ov I}(n)= \triangle P_I(n)+\lm ((I^n:a)/I^{n-1}).$$  
Since $a$ is superficial for $I$, there exists  $c\geq 0$ such that $(I^n:a)\cap I^c=I^{n-1}$ for all $n>c$. Hence the map $\phi :(I^n:a)/I^{n-1}\longrightarrow A/I^c,~ \phi (\ov b)=b+I^c$ has kernel $(I^n:a)\cap I^c/I^{n-1}$ which is $0$ for large $n.$ Hence for large $n,$ $\lm ((I^n:a)/I^{n-1})\leq \lm (A/I^c).$ 
We proceed to prove that for large $n,$ 
$$\lm ((I^n:a)/I^{n-1})=\lm (0:a).$$ 
From the exact sequence 
$$0\longrightarrow A/I^{n-1} \longrightarrow A/I^c \oplus A/(I^n:a)\longrightarrow A/(I^c+(I^n:a))\longrightarrow 0$$
we get $$\lm (A/I^c)+\lm (A/(I^n:a))=\lm (A/I^{n-1})+\lm [A/(I^c+(I^n:a))].$$ 
Hence $\lm ((I^n:a)/I^{n-1})=\lm (I^c+(I^n:a)/I^c).$
Notice that $I^c+(I^n:a)=(0:a)+I^c$ for all $n>c$. Indeed, let $b\in (I^n:a)$. Then
for large $n\geq p,$ by Artin-Rees Lemma, we get,
 $$ba \in (I^n:a)a=I^n\cap (a)=I^{n-p}(I^p\cap (a))\subseteq aI^{n-p}.$$ 
Hence  $ba =ay$ for some $y\in I^{n-p}$. Thus $b-y\in (0:a)$ and so $b\in (0:a)+I^{n-p}\subseteq (0:a)+I^c$ for large $n.$ Hence 
$$\lm((I^n:a)/I^{n-1})=\lm ((0:a)+I^c)/I^c)=\lm ((0:a)/(I^c\cap (0:a))).$$
But $I^c\cap (0:a)\subseteq I^c\cap (I^n:a)=I^{n-1}$ for all large $n.$ 
By Krull Intersection Theorem, $I^c\cap (0:a)=0$. Hence 
$$P_{\ov I}(n)=\triangle P_I(n)+\lm (0:a).$$
Since $\dim A/(a)=\deg ~P_{\ov I}(n)=\deg P_I(n)-1=d-1$, $a$ is a parameter for $A$. The equation above gives $e_0(\ov I)=e_0(I)+\lm (0:a)$ when $d=1$ and for $d\geq 2$, $e_i(\ov I)=e_i(I)$ for $i=0,1,\ldots , d-2$ and $e_{d-1}(\ov I)=e_{d-1}(I)+\lm (0\colon a).$
\end{proof}

\begin{Theorem} [Sally-Machine]
 Let $(A,\mm)$ be a local ring. Let $x$ be superficial for $I$. Suppose 
$\depth G(I/(x))>0$. Then $x^*$ is $G(I)$-regular.
\end{Theorem}

\begin{proof}(B. Singh)
Put $G(I)=\bigoplus_{n=0}^\infty G_n$ where $G_n=I^n/I^{n+1}$. We will show that $(x^*)^s$ is $G(I)$-regular for all $s$. We will use induction on $n$ to show that
for all $s\geq 0, ~ G_n\cap(0:(x^*)^s)=0.$  
Let $f: G(I)\rightarrow G(I/(x))$ be the natural map. Then $f(G_0)=G(I/(x))_0=G_0$. 

Note that $f(0:(x^*)^s)=0$. Indeed, since $x$ is $I$-superficial, $(x^*)^s:\,G_n\longrightarrow G_{n+1}$ is injective for large $n$ and all $s\geq 1$. Hence $(0:(x^*)^s)G_n\subseteq (0:(x^*)^s)\cap G_n=0$ for large  $n.$ Therefore $f(G_n)f(0:(x^*)^s)=0$ for large $n.$ But $f(G_n)=G(I/xA)_n$ has a $G(I)$-regular element for large $n$ as $\depth G(I/xA)> 0$. Thus $f(0:(x^*)^s)=0$.

Let $\overline{a}\in G_0\cap (0:(x^*)^s)$. Then $f(\overline{a})=\overline{a} \in f((0:(x^*)^s))=0$. Suppose that for all $r=\displaystyle{0, 1, \cdots, n-1}$, $G_r \cap (0:(x^*)^s)=0$. Let $b\in I^n \smallsetminus I^{n+1}$ and $b^*\in G_n \cap(0:(x^*)^s)$. Then $b^*(x^*)^s=0$. Hence $bx^s\in I^{n+s+1}$. Since $f(b^*)=0$, $b\in I^{n+1}+xA$. Let $b=c+dx$ for $c\in I^{n+1},\,\,d\in A$. If $d\in I^n$ then $b\in I^{n+1}$ which is a contradiction. So let $d \in I^t \smallsetminus I^{t+1}$, where $t<n$. Since $cx^s=bx^s-dx^{s+1}\in I^{n+s+1}$, we get  $dx^{s+1} \in I^{n+s+1} \subset I^{t+s+1}$. Hence $d^*(x^*)^{s+1}=0$. Thus $d^*\in G_t \cap (0:(x^*)^{s+1})=0$ by induction. This is a contradiction. Hence $d\in I^n$ and therefore $b\in I^{n+1}$, thus $b^*=0$.
\end{proof}

\begin{Proposition}
 Let $(A,\mm)$ be a local ring. Let $I$ be $\mm$-primary. Let $x_1,\,x_2,\,\cdots,\,x_r$ be a superficial sequence for $I.$  Suppose that
$\depth A\geq r.$ Then $x_1,\,x_2,\,\cdots,\,x_r$ is an $A$-regular sequence.
\end{Proposition}

\begin{proof}
 Apply induction on $r$. Let $r=1$. Since $x_1$ is superficial, there is a $c$ such that $(I^n:x_1)\cap I^c=I^{n-1}$ for all $n>c$. Thus $(0:x_1)\cap I^c \subseteq I^{n-1}$ for all $n$ large. By Krull Intersection Theorem, $(0:x_1)\cap I^c=0$. But depth $A$= grade $I >0,$ hence $I^c$ has a regular element, say $a$. Then $(0:x_1)a\subseteq (0:x_1)\cap I^c=0.$ Therefore $(0:x_1)=0$.

As $\overline{x_2},\,\overline{x_3},\,\cdots,\,\overline{x_r}$ is $A/{x_1A}$-superficial sequence and depth $A/{x_1A} \geq r-1$, by induction $\overline{x_2},\,\overline{x_3},\,\cdots,\,\overline{x_r}$ is an $A/{x_1A}$-regular sequence. Hence $\displaystyle{x_1, x_2,\cdots,x_r}$ is an $A$-regular sequence.
\end{proof}

We end this section  by proving an important  criterion due to  Valabrega
and Valla 
\cite{vv}  for a sequence of initial forms $\displaystyle{x_1^*,x_2^*,\cdots, x_s^*}$ in $I/I^2$ to be a $G(I)$-regular sequence.

\begin{Theorem}[Valabrega-Valla, 1978]
 Let $(R,\mm)$ be a local ring. Let  $I$ be an ideal of $R.$ Let $\displaystyle{x_1,x_2,\ldots, x_s}\in I \setminus I^2.$ Then  $\displaystyle{x_1^*,x_2^*,\ldots, x_s^*}$ is a $G(I)$-regular sequence if and only if $\displaystyle{x_1,x_2,\ldots, x_s}$ is an $R$-sequence and  for all $n\geq 1,$  $$(\displaystyle{x_1,x_2,\ldots, x_s})\cap I^n=(\displaystyle{x_1,x_2,\ldots, x_s})I^{n-1}.$$
\end{Theorem}

\begin{proof}
 Apply induction on $s$. Let $s=1$ and put $x_1=x.$ Let $x^*$ be $G(I)$-regular. Let $a\in R$ and $ax=0.$ If $a\not=0$, there is an $m$ such that $a\in I^m \setminus I^{m+1}$. Then $a^*x^*=0$. Hence $a^*=0$, which is a contradiction. Thus $x$ is $R$-regular. Next we show $(x)\cap I^n=xI^{n-1}$ for $n\geq1$. Let $b\in I^m \setminus I^{m+1}$ and $bx\in I^n$. Then $b^*x^*\in I^{m+1}/I^{m+2}$. Since $x^*$ is $G(I)$-regular and $b^*\not=0$, $b^*x^*\not=0$. Thus $bx\not \in I^{m+2}$. Therefore $n-1 \leq m$ and so $b\in I^{n-1}.$

Conversely let $x$ be $R$-regular and $(x)\cap I^n=xI^{n-1}$ for $n\geq1$. Let $b^*\in I^m/I^{m+1}$ and $b^*x^*=0$, then $bx\in I^{m+2}\cap xR=I^{m+1}x$. As $x$ is regular in $R$, $b\in I^{m+1}$. Hence $b^*=0.$

Now assume the result for $s-1,\,s\geq 2$. Let $\displaystyle{x_1^*,x_2^*,\ldots, x_s^*}$ be $G(I)$-regular. Let $S=R/(x_1)$ and $J=I/(x_1)$. Let $``-"$ denote 
images in $S$. Since $x_1^*$ is $G(I)$-regular, $G(I/x_1)\simeq G(I)/(x_1^*)$. Hence $\displaystyle{\overline{x_2}^*, \ldots, \overline{x_s}^*}$ is a $G(I/(x_1))$-regular sequence. By induction hypothesis $\displaystyle{\overline{x_2}, \ldots, \overline{x_s}}$ is $R/(x_1)$-regular sequence and for $n\geq 1,$ 
\begin{eqnarray}
 J^n\cap (\displaystyle{\overline{x_2}, \ldots, \overline{x_s}})=(\displaystyle{\overline{x_2}, \ldots, \overline{x_s}})J^{n-1}.
\end{eqnarray}
Since $x_1^*$ is $G(I)$-regular, $x_1$ is $R$-regular. Hence  $\displaystyle{x_1,x_2,\ldots, x_s}$ is an $R$-regular sequence. We need to prove for $n\geq 1,$
\begin{eqnarray}
 \label{vv} I^n\cap (\displaystyle{x_1,x_2,\ldots, x_s})=(\displaystyle{x_1,x_2,\ldots, x_s})I^{n-1}.
\end{eqnarray}
Let $r_1x_1+\cdots+r_sx_s \in I^n$ for some $\displaystyle{r_1,\ldots, r_s}\in R$. Then
$$\overline{r_2}\overline{x_2}+\cdots + \overline{r_s}\overline{x_s} \in J^n\cap (\displaystyle{\overline{x_2}, \ldots, \overline{x_s}})
= J^{n-1}(\displaystyle{\overline{x_2}, \ldots, \overline{x_s}}).$$
Hence $\overline{r_2}\overline{x_2}+\ldots + \overline{r_s}\overline{x_s}=\overline{t_2}\overline{x_2}+\cdots+\overline{t_s}\overline{x_s}$ for some $\displaystyle{t_1, \ldots,t_s} \in I^{n-1}$. Thus for some $t_1\in R$,

\let\implies=\Longrightarrow
$$(r_2-t_2)x_2+\cdots+(r_s-t_s)x_s=t_1x_1$$
 Hence $$(r_1+t_1)x_1=(r_1x_1+\cdots+r_sx_s)-(t_2x_2+\cdots+t_sx_s)\in I^n.
$$
Therefore $(r_1+t_1) \in I^{n-1}$. This gives $$r_1x_1+\cdots+r_sx_s \in (\displaystyle{x_1,x_2,\ldots, x_s})I^{n-1}.$$

\noindent Conversely let $\displaystyle{x_1,x_2,\cdots, x_s}$ be an $R$-sequence and let (\ref{vv})  hold for all $n\geq 1$. Suppose we prove 
\begin{eqnarray}
\label{vv1} \hspace{4.cm} I^n\cap (\displaystyle{x_1,x_2,\ldots, x_{s-1}})=(\displaystyle{x_1,x_2,\ldots, x_{s-1}})I^{n-1}
\end{eqnarray}
for all $n\geq1$, then by $s-1$ case $\displaystyle{x_1^*,x_2^*,\ldots, x_{s-1}^*}$ is a $G(I)$-regular sequence. Thus $$G(I)/(\displaystyle{x_1^*,x_2^*,\ldots, x_{s-1}^*})\simeq G(I/(\displaystyle{x_1,x_2,\ldots, x_{s-1}}).$$ By $s=1$ case, $x_s^*$ is $G(I)/( \displaystyle{x_1^*,x_2^*,\ldots, x_{s-1}^*})$-regular. Hence $\displaystyle{x_1^*,x_2^*,\ldots, x_s^*}$ is $G(I)$-regular.

We prove (\ref{vv1}) by induction on $n$. The $n=1$ case is clear. Let $n\geq2$ and $r_1s_1+\cdots+r_{s-1}x_{s-1}\in I^n$ for some $\displaystyle{r_1, r_2,\ldots, r_{s-1}}\in R$. Then $$r_1s_1+\cdots+r_{s-1}x_{s-1}\in I^n\cap (\displaystyle{x_1,x_2,\ldots, x_s})=(\displaystyle{x_1,x_2,\ldots, x_s})I^{n-1}.$$ Hence there exist $\displaystyle{t_1, \ldots,t_s} \in I^{n-1}$ such that $$r_1x_1+\ldots+r_{s-1}x_{s-1}=x_1t_1+\ldots+x_st_s.$$ 
Hence $t_sx_s\in (\displaystyle{x_1,x_2,\ldots, x_{s-1}})$. As $\displaystyle{x_1,x_2,\ldots, x_s}$ is an $R$-sequence, $$t_s\in (\displaystyle{x_1,x_2,\ldots, x_{s-1}})\cap I^{n-1}=(\displaystyle{x_1,x_2,\ldots, x_{s-1}})I^{n-2}.$$ Therefore  $t_sx_s\in I^{n-1}(\displaystyle{x_1,x_2,\ldots, x_{s-1}})$. Hence 
$$r_1x_1+\cdots+r_{s-1}x_{s-1}\in I^{n-1}(\displaystyle{x_1,x_2,\ldots, x_{s-1}}).$$
\end{proof}

\section{\bf Huckaba-Marley Theorem}
In this section we prove, by classical techniques, theorems due to Huckaba and 
Marley which characterize  Cohen-Macaulay property of $G(I)$ and  
$\depth G(I) \geq \dim R-1,$  in terms of $e_1(I).$ For another proof of 
these theorems using Rees algebras, we refer the reader to \cite{jsv}. We first prove  
the results in dimension one and then using the  Sally machine, we are able to 
prove them in general. We will also  prove  several consequences of 
these theorems. 
\begin{Theorem}[Huckaba-Marley; 1996, 1997]
Suppose  $(R,\mm)$ is a Cohen-Macaulay local ring of dimension $d$ with $k=R/\mm$ infinite. Let $J$ be a minimal reduction of $I$. Then \\

\hspace*{.5cm}$(1)$~ $\displaystyle{\sum_{n\geq 1}{ \lm (I^n/J\cap I^n)}}\leq e_1(I) \leq \displaystyle {\sum_{n\geq 1}\lm (I^n/JI^{n-1})}.$

\hspace*{.5cm}$(2)$~ $e_1(I)=\displaystyle{\sum_{n\geq 1}\lm (I^n/J\cap I^n)}$ if and only if $G(I)$ is Cohen-Macaulay.

\hspace*{.5cm}$(3)$~ $ e_1(I) = \displaystyle {\sum_{n\geq 1}\lm (I^n/JI^{n-1})}$ if and only if $\depth G(I)\geq d-1.$

\end{Theorem}
\begin{proof}
Apply induction on $d$. The  $d=1$ case is already proved in Theorem \ref{dim1}. 


Now let $d \geq 2.$ Assume the theorem for $d-1$. Let $J=(a_1,\ldots a_d)$ where $a_1,\ldots ,a_d$ is a superficial sequence for $I$. Let ``-'' denote images in $\ov R=R/a_1R$. By induction hypothesis,
\begin{eqnarray*}
e_1(I)=e_1(I/a_1R) & \leq & \displaystyle{\sum_{i=1}^{\infty}\lm (\ov {I}^i/\ov J \ov I^{i-1})}=\displaystyle{\sum_{i=1}^\infty \lm \left(\frac{I^i+a_1R}{JI^{i-1}+a_1R}\right)}\\
& =& \displaystyle{\sum_{i=1}^\infty \lm \left(\frac{I^i}{JI^{i-1}+I^i\cap a_1R}\right)}\\
& \leq & \displaystyle{\sum_{i=1}^\infty \lm (I^i/JI^{i-1})}.
\end{eqnarray*}
By induction hypothesis we also have 
$$e_1(I)=e_1(I/a_1R)\geq \displaystyle{\sum_{i=1}^\infty\lm \left (\frac{\ov I^i+\ov J}{\ov J}\right )}=\displaystyle{\sum_{i=1}^\infty \lm \left (\frac{I^i+J}{J}\right )}.$$
(2) Now we show that $G(I)$ is Cohen-Macaulay if and only if $e_1(I)=\sum_{i=1}^{\infty}\lm (I^i+J/J).$
Let $G(I)$ be Cohen-Macaulay. Let $J=(a_1,\ldots ,a_d)$ be a minimal reduction of $I$. Then $a_1^*,\ldots ,a_d^*$ in $I/I^2$ is a $G(I)-$regular sequence. Thus $G(I/a,R)\cong G(I)/(a^*)$ is Cohen-Macaulay. By induction $$e_1(I)=e_1(I/a_1R)=\displaystyle{\sum_{n=1}^{\infty}\lm \left(\frac{\ov {I}^n+\ov J}{\ov J}\right)}=\displaystyle{\sum_{n=1}^{\infty}\lm \left(\frac{I^n+J}{J}\right)}.$$
Conversely let $e_1(I)=\sum_{n=1}^{\infty}\lm \left(I^n+J/J\right)$. Then $e_1(I/a_1R)=e_1(I)=\sum_{n=1}^{\infty}\lm \left(\ov I^n+\ov J/\ov J\right)$. By induction hypothesis $G(I/a,R)$ is Cohen-Macaulay. By Sally-machine $a_1^*$ is $G(I)$-regular. Therefore $G(I/a_1,R)=G(I)/(a_1^*)$. Hence $\ov a_2^*,\ldots ,\ov a_d^*$ in $G(I)/(a_1^*)$ is a regular sequence. Therefore $G(I)$ is Cohen-Macaulay.\\
(3) Let $e_1(I)=\sum_{n=1}^\infty \lm(I^n/JI^{n-1})$. Let $a_1,\ldots ,a_d$ be an $I$-superficial sequence generating $J$. Let $K=(a_1,\ldots ,a_{d-1})$. Then 
\begin{eqnarray*}
e_1(I)=e_1(I/K)& = &\sum_{n=1}^\infty \lm \left(\frac{I^n+K}{a_dI^{n-1}+K}\right)\\
&=& \sum_{n=1}^\infty \lm \left (\frac{I^n}{a_dI^{n-1}+K\cap I^n}\right)\\
&=& \sum_{n=1}^\infty \lm (I^n/JI^{n-1}).
\end{eqnarray*}
Therefore $a_dI^{n-1}+K\cap I^n=JI^{n-1}$ for all $n\geq 1$. Thus $K\cap I^n \subseteq JI^{n-1}$ for all $n\geq 1$. By the next lemma $a_1^*,\ldots  ,a_{d-1}^*$ is $G(I)-$regular. Conversely let $\depth G(I)\geq d-1$. Let $J=(a_1,\ldots ,a_d)$ be a minimal reduction of $I$ such that $a_1^*\in I/I^2$ is $G(I)-$regular. Then 
\begin{eqnarray*}
e_1(I)=e_1(I/a_1R) & = & \displaystyle{\sum_{n=1}^{\infty} \lm (\ov I^n/\ov J \ov I^{n-1})}\\&=&\displaystyle{\sum_{n=1}^{\infty}\lm \left(\frac{I^n+a_1R}{JI^{n-1}+a_1R}\right)}\\
& = & \displaystyle{\sum_{n=1}^{\infty}\lm \left(\frac{I^n}{JI^{n-1}+a_1R\cap I^n}\right)}\\&=&\displaystyle{\sum_{n=1}^{\infty}\lm \left(\frac{I^n}{JI^{n-1}+a_1I^{n-1}}\right)}\\
& = & \displaystyle{\sum_{n=1}^{\infty}\lm (I^n/JI^{n-1})}.
\end{eqnarray*}
\end{proof}
\begin{Lemma}
 Let $(R,\mm)$ be a $d$-dimensional Cohen-Macaulay local ring and let $I$ be an $\mm-$primary ideal. Suppose $(x_1, x_2, \dots ,x_d)=J$ is a minimal reduction of $I$ such that 
$$I^n\cap (x_1,\ldots ,x_{d-1})\subseteq JI^{n-1}$$
for all $n\geq 1$. Then $x_1^*,\ldots x_{d-1}^*$ is a $G(I)-$regular sequence.
\end{Lemma}
\begin{proof}
 By Valabrega-Valla Theorem it is enough to prove that 
\begin{eqnarray}
\label{6.2} I^n\cap (x_1,\ldots ,x_{d-1})=(x_1,\ldots ,x_{d-1})I^{n-1}
\end{eqnarray}
for all $n\geq 1$. For $n=1,$ (\ref{6.2}) is clearly true. Suppose (\ref{6.2}) is true for $n-1$. To prove (\ref{6.2}) for $n$, let 
$z\in I^n\cap (x_1,\ldots ,x_{d-1})$. Then 
$$z=r_1x_1+\cdots+ r_{d-1}x_{d-1}=s_1x_1+\ldots +s_{d-1}x_{d-1}+px_d,$$  where $r_1,\ldots r_{d-1}\in R$; $~p,s_1,\ldots ,s_{d-1}\in I^{n-1}$. Hence $px_d\in (x_1,\ldots  , x_{d-1})$. Since $x_1,\ldots ,x_{d-1},x_d$ is a regular sequence, $p\in (x_1,\ldots , x_{d-1})$. Therefore 
\begin{eqnarray*}
 I^n\cap (x_1,\ldots ,x_{d-1}) & = & (x_1,\ldots , x_{d-1})I^{n-1}+x_d(I^{n-1}\cap (x_1,\ldots , x_{d-1}))\\
& =& (x_1,\ldots , x_{d-1})I^{n-1}+x_d(x_1,\ldots , x_{d-1})I^{n-2}\\
&=& (x_1,\ldots , x_{d-1})I^{n-1}.
\end{eqnarray*}
Hence $x_1^*,\ldots , x_{d-1}^*$ is a $G(I)-$regular sequence.
\end{proof}

\bigskip 

\noindent{\bf Consequences of Huckaba-Marley Theorem}

\bigskip
The characterizations of depth of $G(I)$ in terms of $e_1(I)$ imply several results obtained by various authors. We assume in this section that $(R,\mm)$ is  Cohen-Macaulay  of dimension $d,$ $I$ is  an $\mm$-primary ideal and  
$R/\mm$  is infinite.
\begin{Corollary}[Northcott, 1960]
 $e_1(I)=0$  if and only if $I$ is generated by a regular sequence.
\end{Corollary}
\begin{proof}
 For  a minimal reduction $J$ of $I,$\; $\sum_{n=1}^{\infty}\lm (I^n/J\cap I^n)\leq e_1(I)=0$. Therefore $\lm(I/J)=0$ which means $I=J$.
\end{proof}
\begin{Corollary}[Sally, 1980]
Let $r(\mm)\leq 2$. Then $G(\mm)$ is Cohen-Macaulay. 
\end{Corollary}
\begin{proof}
 For a  minimal reduction $J$ of $\mm,$ $J\cap \mm^2=\mm J$ since a minimal basis of $J$ can be extended to
a minimal basis of $\mm$. Since $\mm^3=J \mm^2$, for $n\geq 3,$~~ $\mm^n\cap J=J\mm^{n-1}\cap J=J\mm^{n-1}$ and $\mm^n=J \mm^{n-1}$. Therefore
\begin{eqnarray*}
\sum_{n=1}^{\infty}\lm \left(\frac{\mm^n+J}{J}\right)&=&\lm(\mm/J)+\lm(\mm^2/J \mm)\\ 
&\leq& e_1(\mm)\leq \sum_{n=1}^{\infty}\lm \left(\frac{\mm^n}{J \mm^{n-1}}\right)\\&=&\lm(\mm/J)+\lm(\mm^2/J\mm).
\end{eqnarray*}
Thus $e_1(\mm)=\sum_{n=1}^{\infty}\lm \left(m^n/J\cap \mm^n\right)$. Hence $G(\mm)$ is Cohen-Macaulay.
\end{proof}
\begin{Corollary}[Huneke-Ooishi, 1987]
 $e_0(I)-e_1(I)=\lm(R/I)$ if and only if $r(I)\leq 1$. In this case $G(I)$ is Cohen-Macaulay and for all $n\geq 1$, 
$$\lm(R/I^n)=e_0(I) \binom{n+d-1} {d}-e_1(I)\binom{n+d-2} {d-1}.$$
\end{Corollary}
\begin{proof}
Let $e_0(I)-e_1(I)=\lm(R/I)$. Let $J$ be any minimal reduction of $I$. Then $e_0(I)=\lm(R/J)$ so $e_1(I)=\lm(I/J)\geq \sum_{n=1}^{\infty}\lm(I^n/J\cap I^n)$. Therefore $\lm(I^n/J\cap I^n)=0$ for all $n\geq 2$. This implies that $G(I)$ is Cohen-Macaulay.  Hence $I^2=JI$. Let $J=(a_1,\ldots ,a_d)$. Then $a_1^*,\ldots ,a_d^*$  is a $G(I)$-regular sequence . Hence 
\begin{eqnarray*}
G(I)/(a_1^*,\ldots ,a_d^*)\cong G(I/J)&=&\displaystyle{\bigoplus_{n=0}^{\infty}\frac{I^n+J}{I^{n+1}+J}}\\
&=&\frac{R}{I}\bigoplus \frac{I}{I^2+J}\bigoplus \frac{I^2}{I^3+J}\bigoplus \cdots \\&=&\frac{R}{I}\bigoplus \frac{I}{J}.
\end{eqnarray*}
Therefore
\begin{eqnarray*}
H(G(I),t) &= &\displaystyle{\sum_{n=0}^{\infty}\lm(I^n/I^{n+1})t^n}=\frac{h_0+h_1t+\cdots +h_st^s}{(1-t)^d}=\frac{H(G(I/J),t)}{(1-t)^d}\\
& =& \frac{\lm(R/I)+(\lm(R/J)-\lm(R/I))t}{(1-t)^d}
\end{eqnarray*}
Hence $h(t)=\lm(R/I)+[e_0(I)-\lm(R/I)]t$ which gives $e_1(I)=e_0(I)-\lm(R/I)$ and $e_2=e_3=\ldots =e_d=0$ by \cite[Proposition 4.1.9]{bh}. Now we find $\lm(R/I^{n+1}):$
\begin{eqnarray*}
\left(\displaystyle{\sum_{n=0}^{\infty}\lm(R/I^{n+1})t^n}\right)(1-t) & =& \displaystyle{\sum_{n=0}^{\infty}\lm(R/I^{n+1})t^n}-\displaystyle{\sum_{n=0}^{\infty}\lm(R/I^{n+1})t^{n+1}}\\
& = & \displaystyle{\sum_{n=0}^{\infty}\lm(I^n/I^{n+1})t^n}.
\end{eqnarray*}
Hence
\begin{eqnarray*}
\displaystyle{\sum_{n=0}^{\infty}\lm(R/I^{n+1})t^n} & = & \frac{\lm(R/I)+[e_0(I)-\lm(R/I)]t}{(1-t)^{d+1}}\\
& = & [{\lm(R/I)+(e_0(I)-\lm(R/I))t}]\left(\sum_{n=0}^{\infty}
\binom{n+d}{d}t^n\right)
\end{eqnarray*}
Thus for $n \geq 1,$
\begin{eqnarray*}
 \lm(R/I^{n+1}) & = & \lm(R/I)\binom{n+d}{d}+[e_0(I)-\lm(R/I)]\binom{n-1+d}{d}\\
& = & e_0(I)\binom{n+d}{d}-[e_0(I)-\lm(R/I)]\binom{n+d-1}{d-1}.
\end{eqnarray*}

\end{proof}
The next result  is  surprising since it shows that the Cohen-Macaulay 
property  of the Rees algebra $\Rees(I)$ can be determined by  $e_1(I)$ and 
a minimal reduction of $I.$  Goto and Shimoda \cite{gs}  showed that 
$\Rees(I)$ is Cohen-Macaulay if and only if $G(I)$ is Cohen-Macaulay 
and $r(I) \leq d-1.$  
\begin{Corollary}  Let $(R,\mm)$ be a $d$-dimensional Cohen-Macaulay local ring with infinite
residue field. Let $J$ be a minimal reduction of $I.$ Then
$ \Rees(I)$ is  Cohen-Macaulay  if and only if $e_1(I)=\sum_{n=1}^{d-1}\lm (I^n/J\cap I^n).$
\end{Corollary}
\begin{proof}
Let $\Rees(I)$ be Cohen-Macaulay. Hence $G(I)$ is Cohen-Macaulay and 
$r(I) \leq d-1.$ Therefore for any minimal reduction $J$ of $I,$~~ $JI^n=I^{n+1}=J\cap I^n$ for all $n\geq d-1$. Thus $e_1(I)=\sum_{n=1}^{d-1}\lm (I^n/J\cap I^n).$
Conversely, if $e_1(I)=\sum_{n=1}^{d-1}\lm (I^n/J\cap I^n),$ 
then by the inequality $e_1(I)\geq \sum_{n=1}^{\infty}\lm (I^n/J\cap I^n)$ we get $I^n=J\cap I^n$ for all $n\geq d$ which implies $I^n\subseteq J$ for all $n\geq d.$ Moreover, $e_1(I)=\sum_{n=1}^{\infty}\lm (I^n/J\cap I^n)$. Hence $G(I)$ is Cohen-Macaulay.  Therefore $I^d\cap J=JI^{d-1}=I^d$. By Goto-Shimoda Theorem, $\mathcal{R}(I)$ is Cohen-Macaulay.
\end{proof}
\noindent
The next theorem due to A. Guerriere \cite{g} gives a condition  in terms of minimal reduction, for the $\depth G(I)=d-1.$ The orginal proof of this was  rather involved.
\begin{Corollary}[A. Guerriere, 1994]
 Suppose $\sum_{n\geq 2}\lm \left(J\cap I^n/JI^{n-1}\right)=1$. Then 
$\depth \;G(I)=d-1.$
\end{Corollary}
\begin{proof}
We know that $$\sum_{n=1}^{\infty}\lm(I^n/J\cap I^n)\leq e_1(I) \leq \sum_{n=1}^{\infty}\lm(I^n/J I^{n-1}).$$ 
Since $\sum_{n=1}^{\infty}\lm(I^n/J I^{n-1})-\sum_{n=1}^{\infty}\lm(I^n/J\cap I^n)=1$ we conclude $e_1(I)=\sum_{n=1}^{\infty}\lm(I^n/J\cap I^n)$ or $e_1(I)=\sum_{n=1}^{\infty}\lm(I^n/J I^{n-1}).$
The former case does not occur since otherwise $G(I)$ will be Cohen-Macaulay and consequently $I^n\cap J=JI^{n-1}$ for all $n\geq 1.$  Hence    $\sum_{n=1}^{\infty}\lm(J\cap I^n/J I^{n-1})=0,$  which is a contradiction.
Therefore $\depth G(I)=d-1.$
\end{proof}

\end{document}